\documentclass[reqno]{amsart}
\usepackage{amsmath,amsthm,amscd,amssymb,amsfonts, amsbsy}
\usepackage[usenames, dvipsnames]{color}
\usepackage{pxfonts}
\usepackage{bbm}
\usepackage{enumerate}
\usepackage{latexsym}
\usepackage{mathrsfs}
\usepackage{xparse}
\usepackage{marginnote}
\usepackage{todonotes}
\usepackage{tikz}

\numberwithin{equation}{section}

\usepackage{verbatim}

\theoremstyle{plain}
\newtheorem{theorem}{Theorem}[section]
\newtheorem{lemma}[theorem]{Lemma}

\theoremstyle{definition}
\newtheorem{definition}[theorem]{Definition}

\theoremstyle{remark}

\DeclareMathOperator{\tr}{tr}

\providecommand{\set}[1]{\{#1\}}

\providecommand{\abs}[1]{\lvert#1\rvert}
\providecommand{\Abs}[1]{\left\lvert#1\right\rvert}

\newcommand{\bR}{\mathbb{R}}

\renewcommand{\vec}[1]{\boldsymbol{#1}}

\begin{document}


\title[Two-sided Gaussian estimates for fundamental solutions]{Two-sided Gaussian estimates for fundamental solutions of second-order parabolic equations in non-divergence form}

\author[S. Kim]{Seick Kim}
\address[S. Kim]{Department of Mathematics, Yonsei University, 50 Yonsei-Ro, Seodaemun-gu, Seoul 03722, Republic of Korea.}
\email{kimseick@yonsei.ac.kr}
\thanks{Kim is supported by the National Research Foundation of Korea (NRF) under agreement NRF-2022R1A2C1003322.}

\author[S. Lee]{Sungjin Lee}
\address[S. Lee]{Aristotle University of Thessaloniki,
School of Mathematics, 541 24 Thessaloniki, Greece}
\email{slee@math.auth.gr }

\author[G. Sakellaris]{Georgios Sakellaris}
\address[G. Sakellaris]{Aristotle University of Thessaloniki,
School of Mathematics, 541 24 Thessaloniki, Greece}
\email{gsakell@math.auth.gr}
\thanks{G. Sakellaris acknowledges that this project is carried out within the framework of the National Recovery and Resilience Plan Greece 2.0, funded by the European Union, NextGenerationEU (Implementation body: HFRI, Project Name: CRITPDE, No. 14952).}


\begin{abstract}
We establish two-sided Gaussian bounds for the fundamental solution of second-order parabolic operators in non-divergence form under minimal assumptions. Specifically, we show that the upper and lower bounds follow from the local boundedness property and the weak Harnack inequality for the adjoint operator $P^*$, respectively. This provides a simpler and more direct proof of the Gaussian estimates when the coefficients have Dini mean oscillation in  $x$, avoiding the use of normalized adjoint solutions required in previous works.
\end{abstract}

\maketitle

\section{Introduction and main results}

We consider the second-order parabolic operator $P$ in non-divergence form:
\[
Pu=\partial_t u -a^{ij}(t,x)D_{ij}u.
\]
The formal adjoint of $P$, denoted by $P^*$ is given by
\[
P^*v=-\partial_t v -D_{ij}(a^{ij}(t,x)v).
\]

We assume that the coefficient matrix $\mathbf{A}=(a^{ij})$ is symmetric and satisfies the uniform parabolicity condition:
\begin{equation}			\label{parabolicity}
\lambda \abs{\vec \xi}^2\leq a^{ij}(t,x)\xi_i\xi_j\leq \Lambda \abs{\vec \xi}^2,\quad\forall \vec \xi\in\bR^d,\quad \forall(t,x)\in \bR^{d+1},
\end{equation}
for some constants $0<\lambda \le \Lambda <\infty$.

A fundamental solution $\Gamma(t,x,s,y)$ for $P$ is a function satisfying
\[
P  \Gamma(\cdot, \cdot,s,y)=\delta_{{s,y}}(\cdot, \cdot)\;\text{ in }\;\bR^{d+1},
\]
or equivalently,
\[
P\Gamma(\cdot, \cdot, s,y) =0\;\text{ in }\; (s,\infty )\times \bR^d,\quad \lim_{t\searrow s}\Gamma(t,\cdot, s,y)=\delta_y(\cdot)\;\text{ on }\;\bR^d.
\]

By two-sided Gaussian estimates for $\Gamma(t,x,s,y)$, we mean the following bounds: for any $x$,$y\in \bR^d$ and $t$, $s\in \bR$ satisfying $0<t-s<T$,
\begin{equation} \label{eq:TG}
\frac{1}{N_1(t-s)^{d/2}} \exp \left\{-\kappa_1 \frac{\abs{x-y}^2}{t-s}\right\} \le \Gamma(t,x,s,y)\le \frac{N_1}{(t-s)^{d/2}} \exp \left\{-\frac{\abs{x-y}^2}{\kappa_1(t-s)}\right\},
\end{equation}
where $\kappa_1>0$ and $N_1>0$ are constants. The constant $N_1$ may depend on $T$.

It is well known that for parabolic operators in divergence form,
\[
\mathcal{P}u=\partial_t u - D_i(a^{ij}D_ju),
\]
the fundamental solution satisfies the two-sided Gaussian bounds \eqref{eq:TG} with $T=\infty$, where $\kappa_1$ and $N_1$ depend only on the dimension $d$ and the ellipticity constants $\lambda$ and $\Lambda$.
This was first established by Aronson \cite{Aronson}, based on Moser's Harnack inequality \cite{Moser}.
Later, Fabes and Stroock \cite{FS86} showed that Nash's ideas \cite{Nash} could be used to derive Aronson's bounds, thereby providing an alternative proof of Moser's parabolic Harnack inequality.

In contrast, for parabolic equations in non-divergence form, the fundamental solution does not generally satisfy Gaussian bounds unless additional regularity assumptions are imposed on the coefficients.
In fact, even if the coefficients are continuous in $t$ and $x$, the weaker upper bound
\[
\Gamma(t,x,s,y) \le \frac{C}{(t-s)^{d/2}}\quad \text{for }\,0<t-s<T,
\]
may still fail.
A counterexample is provided in \cite{Ilin62} for the one-dimensional equation $u_t -a(t,x) u_{xx}=0$, where the coefficient $a(t,x)$, though continuous and bounded by $1/2 \le a(t,x) \le 3/2$, yields a fundamental solution that does not satisfy that does not satisfy the above estimate.
See also \cite{FK81, Sa81} for further examples where the fundamental solution, as a measure, is singular with respect to Lebesgue measure for any $t > 0$, even under the continuity of coefficients.

When the coefficient matrix $\mathbf A$ is H\"older or Dini continuous, the two-sided bounds \eqref{eq:TG} do hold; see, for instance, \cite{Friedman, PE84, IO1962, LSU}.
In the case of merely continuous coefficients, Escauriaza \cite{Esc2000} established bounds for the fundamental solution involving a nonnegative function $W$, which is a parabolic Muckenhoupt weight belonging to the reverse H\"older class $\mathcal{B}_{(d+1)/d}$ and satisfies $P^*W=0$ in $\bR^{d+1}$.
This is a significant result, serving as a counterpart to Aronson’s classical estimates for divergence form equations. However, the function $W$ can be quite singular, and it is not clear under what conditions one can dispense with $W$ and recover the standard bounds \eqref{eq:TG}.
In \cite{DEK21}, it is shown that if the coefficient matrix $\mathbf A$ has Dini mean oscillation in the spatial variable $x$ (see Definition \ref{dmox}), then the weight $W$ is regular, and the Gaussian bounds \eqref{eq:TG} are valid.
In particular, the arguments in \cite{DEK21} implicitly show that the function $W$ satisfies the Harnack inequality for the formal adjoint operator $P^*$: if $u$ is a nonnegative solution of $P^*u=0$ in $(t_0, t_0+4r^2)\times B_{2r}(x_0)$, where $0<r<R_0$, then
\[
\sup_{(t_0+2r^2, t_0+3r^2)\times B_{r}(x_0)} u \le N \inf_{(t_0, t_0+r^2)\times B_{r}(x_0)} u,
\]
where $N>0$ is a constant depending only on $d$, $\lambda$, $\Lambda$, the Dini mean oscillation modulus of $\mathbf A$ in $x$, and $R_0$.
We emphasize that the time directions for the Harnack inequalities associated with $P$ and $P^*$ are opposite.

As noted earlier, in the divergence form case, the two-sided bounds \eqref{eq:TG} are intimately related to the Harnack inequality. According to the Krylov--Safonov theorem \cite{KS}, nonnegative solutions to $Pu=0$ satisfy the Harnack inequality.
However, in the non-divergence case, the formal adjoint operator $P^*$ does not necessarily inherit such properties from $P$.
In particular, the Harnack inequality may fail for $P^*$ unless additional regularity assumptions are imposed on $\mathbf A$.

This leads to a natural question: if the Harnack inequality holds for $P^*$, can one deduce the two-sided Gaussian bounds \eqref{eq:TG}?
We answer this question affirmatively.
Specifically, we show that if $P^*$ satisfies the local boundedness property, then the upper bound in \eqref{eq:TG} follows.
Moreover, if $P^*$ satisfies the weak Harnack inequality, then the lower bound in \eqref{eq:TG} holds. 

\begin{definition}[Local Boundedness Property]	\label{def_lb}
We say that $P^*$ satisfies the local boundedness property for solutions if, for any solution $u$ of $P^*u=0$ in 
\[
C_{2r}(X_0)=(t_0, t_0+4r^2)\times B_{2r}(x_0),
\]
where $X_0=(t_0,x_0) \in \bR^{d+1}$ and $0<r<R_0$, the estimate
\begin{equation}			\label{eq2134thu}
\sup_{C_r(X_0)} \abs{u} \le N_0 \fint_{C_{2r}(X_0)} \abs{u(t,x)}\,dx\,dt,
\end{equation}
holds.
Here, $N_0>0$ may depend on $R_0$, but it is uniform for all $r \in (0,R_0)$ and $X_0 \in \bR^{d+1}$.
In particular, $N_0$ is independent of $u$.
\end{definition}

\begin{definition}[Weak Harnack Inequality]	\label{def_wh}
We say that $P^*$ satisfies the weak Harnack inequality for solutions if, for any nonnegative solution $u$ of $P^*u=0$ in $C_{2r}(X_0)$, where $X_0=(t_0,x_0) \in \bR^{d+1}$ and $0<r<R_0$, the estimate
\begin{equation}			\label{eq2135thu}
N_0 \inf_{C_r(X_0)} u \ge \fint_{t_0+2r^2}^{t_0+3r^2}\!\!\! \fint_{B_r(x_0)} u(t,x)\,dx\,dt,
\end{equation}
holds, with the same uniformity assumptions on $N_0$ as above.
\end{definition}
 
The constants $N_0$ in the above definitions are not necessarily identical; however, for simplicity, we assume they are the same throughout this article.

\begin{theorem}		\label{thm1}
Suppose that the formal adjoint operator $P^*$ satisfies the local boundedness property.
Given $T>0$, let $N_0$ be the constant associated with $R_0=\sqrt{T}$ in Definition~\ref{def_lb}.
Then the fundamental solution $\Gamma(t,x,s,y)$ of $P$ satisfies the upper Gaussian bound: for any $x$,$y\in \bR^d$ and $t$, $s\in \bR$ such that $0<t-s<T$, we have
\[
\Gamma(t,x,s,y) \le \frac{N}{(t-s)^{d/2}} \exp\left\{-\kappa\frac{\abs{x-y}^2}{t-s}\right\},
\]
where $\kappa=\kappa(\Lambda)>0$ and $N=N(d,\lambda, \Lambda, N_0)>0$.
\end{theorem}

\begin{theorem}		\label{thm2}
Suppose that the formal adjoint operator $P^*$ satisfies  the weak Harnack inequality.
Given $T>0$, let $N_0$ be the constant associated with $R_0=\sqrt{T}$  in Definition~\ref{def_wh}.
Then the fundamental solution $\Gamma(t,x,s,y)$ of $P$ satisfies the lower Gaussian bound: for any $x$,$y\in \bR^d$ and $t$, $s\in \bR$ such that $0<t-s<T$, we have
\[
\Gamma(t,x,s,y) \ge \frac{1}{N(t-s)^{d/2}} \exp\left\{-\frac{\abs{x-y}^2}{\kappa(t-s)}\right\},
\]
where $\kappa=\kappa(d, \lambda, \Lambda)>0$ and $N=N(d,\lambda, \Lambda, N_0)>0$.
\end{theorem}

The proofs of Theorems \ref{thm1} and \ref{thm2} are given in Sections \ref{sec2} and \ref{sec3}, respectively.

It is straightforward to verify that the combination of the local boundedness property and the weak Harnack inequality implies the Harnack inequality, and vice versa.
This was, in fact, Moser's original approach in \cite{Moser}: he first established local boundedness and then proved the weak Harnack inequality.

Therefore, Theorems \ref{thm1} and \ref{thm2} together yield an alternative proof of the two-sided Gaussian bounds \eqref{eq:TG} under the assumption that $\mathbf A$ has Dini mean oscillation in $x$, as established in \cite{DEK21}.

\begin{definition}			\label{dmox}
Let $\mathbb C_r$ be the collection of all cylinders $C_r(X_0)$ in $\bR^{d+1}$.
For $C=C_r(X_0) \in \mathbb C_r$, we set
\[
\bar{\mathbf A}^{\textsf x}_{C}(t)=\fint_{B_r(x_0)} \mathbf A(t,x)\,dx,
\]
and for $r>0$, we define
\[
\omega_{\mathbf A}^{\textsf x}(r):=\sup_{C \in \mathbb{C}_r}\fint_{C} \,\Abs{\mathbf A(t,x)- \bar{\mathbf A}^{\textsf x}_{C}(t)}\,dxdt.
\]
We say that $\mathbf{A}$ has Dini mean oscillation in $x$ (denoted $\mathbf A \in {\rm DMO}_x$) if
\[
\int_0^{1} \frac{\omega_{\mathbf A}^{\textsf x}(t)}{t}\,dt <\infty.
\]
It is clear that if $\mathbf A$ is Dini continuous in $x$, then $\mathbf A \in {\rm DMO}_x$.
\end{definition}

\begin{theorem}[Dong--Escauriaza--Kim \cite{DEK21}]		\label{thm3}
Suppose that the coefficient matrix $\mathbf{A}$ has Dini mean oscillation in $x$.
Then the fundamental solution $\Gamma(t,x,s,y)$ of $P$ satisfies the two-sided Gaussian bounds \eqref{eq:TG}.
\end{theorem}

The proof of Theorem \ref{thm3} follows directly from the Harnack inequality for the adjoint operator $P^*$ (see \cite[Theorem 2.20]{GK24}), combined with Theorems \ref{thm1} and \ref{thm2}.
We remark that although \cite[Theorem 2.20]{GK24} is not explicitly stated in \cite{DEK21}, it can be inferred from the results therein.
In fact, Theorem \ref{thm3} appears as a remark in \cite[Remark 5.12]{DEK21}, where the results of Escauriaza \cite{Esc2000} play a central role.
Our approach, by contrast, is more direct, as it avoids the use of normalized adjoint solutions, which are essential in \cite{Esc2000}.

We also note that in \cite{DKL22}, the upper Gaussian estimate was established under the stronger assumption that the coefficient matrix $\mathbf{A}$ is Dini continuous in $x$.
However, the method used there is applicable to parabolic systems in non-divergence form, as it does not rely on the Harnack inequality.

\section{Proof of Theorem \ref{thm1}} \label{sec2}
Let $X=(t,x)$, $Y=(s,y)$ be points in $\bR^{d+1}$ with $0<t-s<T$, where $T$ is a fixed number but arbitrary.
Define
\[
r=\sqrt{t-s}/2
\]
and consider the parabolic cylinder
\[
C_{2r}(Y)=(s,s+4r^2)\times B_{2r}(y)=(s,t)\times B_{\sqrt{t-s}}(y).
\]

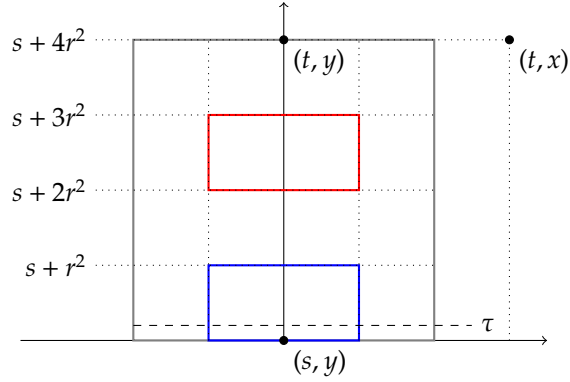
\begin{center}
\begin{tikzpicture}[scale=1]
\draw[->] (-3.5, -4) -- (3.5, -4); 
\draw[->] (0, -4) -- (0, 0.5);
\draw[thick, gray] (-2,-4) rectangle (2,0);
\draw[thick, red] (-1,-2) rectangle (1,-1);
\draw[thick, blue] (-1,-4) rectangle (1,-3);
\filldraw (0,-4) circle (1.5pt) node[below right] {$(s,y)$};
\filldraw (0,0) circle (1.5pt) node[below right] {$(t,y)$};
\filldraw (3,0) circle (1.5pt) node[below right] {$(t,x)$};
\draw[dotted] (-1,-4) -- (-1,0);
\draw[dotted] (1,-4) -- (1,0);
\draw[dotted] (3,-4) -- (3,0);
\draw[dotted] (-2.5, 0) -- (3, 0);
\node[left] at (-2.5, 0) {$s+4r^2$};
\draw[dotted] (-2.5, -1) -- (2, -1);
\node[left] at (-2.5, -1) {$s +3r^2$};
\draw[dotted] (-2.5, -2) -- (2, -2);
\node[left] at (-2.5, -2) {$s + 2r^2$};
\draw[dotted] (-2.5, -3) -- (2, -3);
\node[left] at (-2.5, -3) {$s + r^2$};
\draw[dashed] (-2, -3.8) -- (2.5, -3.8);
\node[right] at (2.5, -3.8) {$\tau$};
\end{tikzpicture}
\end{center}

For each $\tau \in [s, s+r^2]$, consider the following integral on $\{\tau\} \times B_r(y)$.
\begin{equation}			\label{eq0714wed}
\int_{B_r(y)} \Gamma(t,x,\tau,z)\,dz.
\end{equation}

The following lemma, stated in \cite[Lemma 3.1]{Esc2000}, provides a Gaussian decay estimate for the integral in \eqref{eq0714wed}.

\begin{lemma}		\label{lemma02142fri}
Let $x$, $y\in \bR^d$ and $\tau$, $t\in \bR$ with $\tau<t$.
Then, for any $r>0$, we have
\begin{equation}		\label{eq02131thr}
\int_{B_r(y)} \Gamma(t,x,\tau,z)\,dz \le \left( 1+ \frac{t-\tau}{r^2} \right)^{-d\lambda/2\Lambda}\exp\left\{ \frac{1}{4\Lambda} \left(1-\frac{\abs{x-y}^2}{t-\tau+r^2} \right)\right\}.
\end{equation}
\end{lemma}
\begin{proof}
Consider the function
\[
u(t,x):=\int_{B_r(y)} \Gamma(t,x,\tau,z)\,dz.
\]
This function satisfies the equation
\[
Pu=0 \;\text{ in }\; (\tau,\infty) \times \bR^d,\quad u=\mathbbm{1}_{B_r(y)}\;\text{ on }\;\set{t=\tau}\times \bR^d.
\]
Now, define the auxiliary function
\[
h(t,x):=\left( 1+ \frac{t-\tau}{r^2} \right)^{-d\lambda/2\Lambda} \exp\left\{ \frac{1}{4\Lambda} \left(1-\frac{\abs{x-y}^2}{t-\tau+r^2}\right)\right\}.
\]
A direct computation gives
\[
Ph=\left\{ \frac{1}{4\Lambda}\frac{\abs{x-y}^2}{(t-\tau+r^2)^2}-\frac{d\lambda}{2\Lambda(t-\tau+r^2)} -\frac{ a^{ij}(x_i-y_i)(x_j-y_j)}{4\Lambda^2(t-\tau+r^2)^2} +\frac{\tr \mathbf A}{2\Lambda(t-\tau+r^2)}\right\}h \ge 0,
\]
where we used the parabolicity condition \eqref{parabolicity}.
Since $h\ge u$ on $\set{t=\tau}\times \bR^d$, the desired inequality \eqref{eq02131thr} follows from the maximum principle for $P$.
\end{proof}

Applying Lemma \ref{lemma02142fri} with $r=\sqrt{t-s}/2$, we obtain
\begin{equation}	\label{eq0754wed}
\int_{B_{\sqrt{t-s}/2}(y)}\Gamma(t,x,\tau,z)\,dz \le \frac{e^{1/4\Lambda}}{4^{d\lambda/2\Lambda}} \, \exp\left\{-\frac{\abs{x-y}^2}{5\Lambda(t-s)}\right\}\quad \text{for all }\; \tau \in [s, (3s+t)/4].
\end{equation}
Taking average on both sides of \eqref{eq0754wed} over $\tau \in [s, (3s+t)/4]$ and dividing by the volume of $B_{\sqrt{t-s}/2}(y)$, we obtain

\begin{equation*}	
\fint_{s}^{\frac{3s+t}{4}} \!\!\!\fint_{B_{\sqrt{t-s}/2}(y)}\Gamma(t,x,\tau,z) \,dz d\tau\le \frac{2^d e^{1/4\Lambda}}{4^{d\lambda/2\Lambda}\abs{B_1}} \,\frac{1}{(t-s)^{d/2}} \exp\left\{-\frac{\abs{x-y}^2}{5\Lambda(t-s)}\right\}.
\end{equation*}
This, combined with the local boundedness property for $P^*$ in Definition \ref{def_lb}, provides an upper bound for the fundamental solution
\[
\Gamma(t,x,s,y) \le \frac{N_0 2^d e^{1/4\Lambda}}{4^{d\lambda/2\Lambda}\abs{B_1}} \,\frac{1}{(t-s)^{d/2}} \exp\left\{-\frac{\abs{x-y}^2}{5\Lambda(t-s)}\right\},
\]
where $N_0$ may depend on $R_0=\sqrt{T}$.
The theorem is proved.
\qed

\section{Proof of Theorem \ref{thm2}}		\label{sec3}
Let $X=(t,x)$, $Y=(s,y)$ be points in $\bR^{d+1}$ with $0<t-s<T$, where $T$ is a fixed number but arbitrary.
Define
\[
r=\sqrt{t-s}/2
\]
and consider the parabolic cylinder
\[
C_{2r}(Y)=(s,s+4r^2)\times B_{2r}(y)=(s,t)\times B_{\sqrt{t-s}}(y).
\]

\begin{center}
\begin{tikzpicture}[scale=1]
\draw[->] (-3.5, -4) -- (3.5, -4);
\draw[->] (0, -4) -- (0, 0.5);
\draw[thick, gray] (-2,-4) rectangle (2,0);
\draw[thick, red] (-1,-2) rectangle (1,-1);
\draw[thick, blue] (-1,-4) rectangle (1,-3);
\filldraw (0,-4) circle (1.5pt) node[below right] {$(s,y)$};
\filldraw (0,0) circle (1.5pt) node[below right] {$(t,y)$};
\filldraw (3,0) circle (1.5pt) node[below right] {$(t,x)$};
\draw[dotted] (-1,-4) -- (-1,0);
\draw[dotted] (1,-4) -- (1,0);
\draw[dotted] (3,-4) -- (3,0);
\draw[dotted] (-2.5, 0) -- (3, 0);
\node[left] at (-2.5, 0) {$s+4r^2$};
\draw[dotted] (-2.5, -1) -- (2, -1);
\node[left] at (-2.5, -1) {$s +3r^2$};
\draw[dotted] (-2.5, -2) -- (2, -2);
\node[left] at (-2.5, -2) {$s + 2r^2$};
\draw[dashed] (-2, -1.8) -- (2.5, -1.8);
\node[right] at (2.5, -1.8) {$\sigma$};
\draw[dotted] (-2.5, -3) -- (2, -3);
\node[left] at (-2.5, -3) {$s + r^2$};
\end{tikzpicture}
\end{center}

For $\sigma \in [s+2r^2, s+3r^2]$, consider the function
\begin{equation*}			
u^{(\sigma)}(\tau,\xi)=\int_{B_r(y)} \Gamma(\tau,\xi,\sigma,z)\,dz.
\end{equation*}

We recall the following lemma by Safonov and Yuan \cite[Lemma 4.1]{SY99}.
\begin{lemma}[Safonov--Yuan \cite{SY99}]		\label{thm_sy}
There exists a constant $N_1=N_1(d,\lambda, \Lambda)>0$ such that
\[
N_1 \int_{B_{2\rho}(y)} \Gamma(\tau,\xi,\sigma, z)\,dz \ge 1
\]
for all $(\sigma, y) \in \bR^{d+1}$, $\rho>0$, and $(\tau,\xi) \in [\sigma,\sigma+\rho^2]\times B_\rho(y)$.
\end{lemma}

In particular, Theorem \ref{thm_sy} implies
\begin{equation}				\label{eq1412mon}
u^{(\sigma)}(\sigma+r^2/4, y) \ge 1/N_1.
\end{equation}

The following lemma establishes exponential lower bounds and it is based on an iterative application of the Krylov--Safonov Harnack inequality.

\begin{lemma}				\label{lem1106wed}
For $\sigma \in [s+2r^2, s+3r^2]$, we have
\begin{equation}			\label{eq1103wed}
u^{(\sigma)}(t,x) \ge \frac{1}{N_2} \exp\left\{-\kappa_0\,\frac{\abs{x-y}^2}{t-s} \right\},
\end{equation}
where $N_2=N_2(d,\lambda, \Lambda)>1$ and $\kappa_0=\kappa_0(d,\lambda, \Lambda)>0$.
\end{lemma}
\begin{proof}
Define
\begin{equation}			\label{eq1359wed}
k=\left\lceil 4\abs{x-y}^2/(t-\sigma-r^2/4)\right\rceil.
\end{equation}
We consider two cases:

\medskip
\noindent
\textbf{Case 1:} $k \ge 3$.
Define
\[
x_0=y,\quad t_0 = \sigma+r^2/4,
\]
and for $j=1,\ldots,k$, set
\[
x_j=y+j(x-y)/k, \quad t_j=t_0+j(t-t_0)/k.
\]
Note that
\[
t_j-t_{j-1}=(t-t_0)/k,\qquad \abs{x_j-x_{j-1}} =\abs{x-y}/k.
\]
Consider the cylinder
\[
C^{(j)}=B_{2\varepsilon}(y_j)\times (t_{j-1}-\varepsilon^2, t_j),\quad\text{where }\; y_j:=\tfrac{1}{2}(x_{j-1}+x_j)\;\text{ and }\;\varepsilon:=\tfrac{1}{2} \sqrt{(t-t_0)/k}.
\]

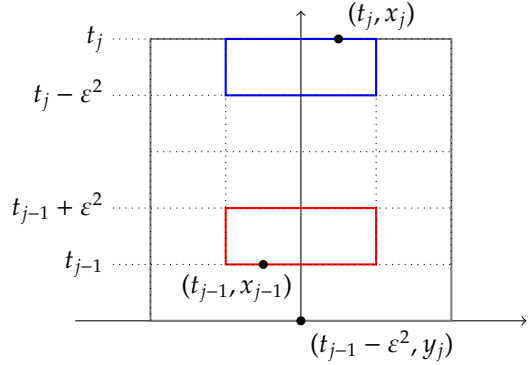
\begin{center}
\begin{tikzpicture}[scale=1]
\draw[->] (-3, -3.75) -- (3, -3.75);
\draw[->] (0, -3.75) -- (0, 0.375);
\draw[thick, gray] (-2,-3.75) rectangle (2,0);
\draw[thick, blue] (-1,-0.75) rectangle (1,0);
\draw[thick, red] (-1,-3) rectangle (1,-2.25);
\filldraw (0,-3.75) circle (1.5pt) node[below right] {$(t_{j-1}-\varepsilon^2,y_j)$};
\filldraw (0.5,0) circle (1.5pt) node[above right] {$(t_j,x_j)$};
\filldraw (-0.5,-3) circle (1.5pt) node[below] {$(t_{j-1},x_{j-1})\qquad$};
\draw[dotted] (-2,-3.75) -- (-2,0);
\draw[dotted] (2,-3) -- (2,0);
\draw[dotted] (-1,-3) -- (-1,0);
\draw[dotted] (1,-3) -- (1,0);
\draw[dotted] (-2.5, 0) -- (2, 0);
\node[left] at (-2.5, 0) {$t_j$};
\draw[dotted] (-2.5, -0.75) -- (2, -0.75);
\node[left] at (-2.5, -0.75) {$t_j-\varepsilon^2$};
\draw[dotted] (-2, -1.5) -- (2, -1.5);
\draw[dotted] (-2.5, -2.25) -- (2, -2.25);
\node[left] at (-2.5, -2.25) {$t_{j-1}+\varepsilon^2$};
\draw[dotted] (-2.5, -3) -- (2, -3);
\node[left] at (-2.5, -3) {$t_{j-1}$};
\end{tikzpicture}
\end{center}

Since $k$ is defined by the formula \eqref{eq1359wed}, we have
\[
\abs{x_j-x_{j-1}} \le \tfrac{1}{2}\sqrt{t_j-t_{j-1}}=\varepsilon,
\]
which in turn implies
\[
\abs{x_j-y_j}=\abs{x_{j-1}-y_j}=\tfrac{1}{2}\abs{x_j-x_{j-1}} \le \tfrac{1}{2} \varepsilon<\varepsilon.
\]
Moreover, the assumption $k\ge 3$ ensures that $C^{(j)} \subset (\sigma,\infty)\times \bR^d$.
Since $u^{(\sigma)} \ge 0$ and satisfies
\begin{equation}			\label{eq1824wed}
P u^{(\sigma)}=0 \quad\text{in}\quad(\sigma,\infty) \times \bR^d,
\end{equation}
the Krylov-Safonov Harnack inequality gives
\[
N_2 \inf_{B_\varepsilon(y_j)\times (t_j-\varepsilon^2,t_j)} u^{(\sigma)} \ge  \sup_{B_\varepsilon(y_j)\times (t_{j-1},t_{j-1}+\varepsilon^2)} u^{(\sigma)},
\]
where $N_2=N_2(d,\lambda, \Lambda)>1$.
In particular, this yields
\begin{equation}			\label{eq1429wed}
u^{(\sigma)}(t_j, x_j) \ge u^{(\sigma)}(t_{j-1}, x_{j-1})/N_2.
\end{equation}
Applying \eqref{eq1429wed} iteratively for $j=1,\ldots, k$ and using \eqref{eq1359wed}, we obtain
\[
u^{(\sigma)}(t, x) \ge N_2^{-k} u^{(\sigma)}(t_0, y) \ge N_2^{-1}\exp\left\{-4\log N_2 \frac{\abs{x-y}^2}{t-t_0} \right\} u^{(\sigma)}(t_0, y).
\]
This, together with \eqref{eq1412mon}, yields \eqref{eq1103wed} since $t_0=\sigma+r^2/4 \in [s+9r^2/4, s+13r^2/4]$.

\medskip
\noindent
\textbf{Case 2:} $k \le 2$.
In this case, \eqref{eq1359wed} implies
\begin{equation}		\label{eq1832wed}
\abs{x-y}^2 \le (t-\sigma-r^2/4)/2 \le  7(t-s)/16.
\end{equation}
Since $u^{(\sigma)}\ge 0$ and satisfies \eqref{eq1824wed}, the Krylov--Safonov Harnack inequality implies that for any $y_0 \in \bR^d$, $s_0\ge \sigma$, and $\rho>0$,
\[
N \inf_{(s_0+3\rho^2, s_0+4\rho^2)\times B_\rho(y_0)} u^{(\sigma)} \ge \sup_{(s_0+\rho^2, s_0+2\rho^2)\times B_\rho(y_0)} u^{(\sigma)}.
\]
Applying the Harnack inequality twice, while considering \eqref{eq1832wed}, we obtain
\[
N^2 u^{(\sigma)}(t,x) \ge u^{(\sigma)}(\sigma+r^2/4,y).
\]
Together with \eqref{eq1412mon}, this yields \eqref{eq1103wed}.
\end{proof}

Next, applying the weak Harnack inequality for $P^*$ in Definition \ref{def_wh}, we obtain
\begin{align*}
\Gamma(t,x,s,y) &\ge \frac{1}{N_0} \frac{1}{r^d \abs{B_1}} \fint_{s+2r^2}^{s+3r^2} \!\!\!\int_{B_r(y)} \Gamma(t,x,\sigma, z) \,dz d\sigma\\
&= \frac{4^d}{N_0(t-s)^{d/2} \abs{B_1}} \fint_{s+2r^2}^{s+3r^2} u^{(\sigma)}(t,x)\, d\sigma,
\end{align*}
where $N_0>0$ may depend on $R_0=\sqrt{T}$.
Applying Lemma \ref{lem1106wed}, we obtain the lower bound
\[
\Gamma(t,x,s,y)\ge \frac{4^d}{N_0N_2\abs{B_1}} \frac{1}{(t-s)^{d/2}}\exp\left\{-\kappa_0\,\frac{\abs{x-y}^2}{t-s} \right\}.
\]
The theorem is proved.\qed



\end{document}